\documentclass{amsart}

\usepackage{amsmath,amsfonts,amssymb,amsthm,pb-diagram,picinpar,graphicx,color}
\usepackage{colortbl}

\newcommand{\bC}{\mathbb{C}}

\newcommand{\bP}{\mathbb{P}}

\newcommand{\bZ}{\mathbb{Z}}

\newcommand{\calA}{\mathcal{A}}

\newcommand{\calF}{\mathcal{F}}
\newcommand{\calI}{\mathcal{I}}
\newcommand{\calL}{\mathcal{L}}

\newcommand{\calO}{\mathcal{O}}

\newcommand{\Gal}{\mathrm{Gal}}
\newcommand{\I}{\mathrm{I}}

\newcommand{\Sing}{\mathrm{Sing}}

\newcommand{\Supp}{\mathrm{Supp}}
\newcommand{\Pic}{\mathrm{Pic}}

\newtheorem{Th}{Theorem}[section]
\newtheorem{Prop}[Th]{Proposition}
\newtheorem{Lem}[Th]{Lemma}
\newtheorem{Cor}[Th]{Corollary}
\newtheorem{Rem}[Th]{Remark}

\newtheorem{Def}[Th]{Definition}

\newtheorem{Problem}[Th]{Problem}

\makeatletter
 
 \@addtoreset{figure}{section}
\makeatother

\begin{document}

\title[splitting numbers for Galois covers]
{A note on splitting numbers for Galois covers and $\pi_1$-equivalent Zariski $k$-plets}
\author{Taketo Shirane}
\address{National Institute of Technology, Ube College, 
2-14-1 Tokiwadai, Ube 755-8555, Yamaguchi Japan}
\email{tshirane@ube-k.ac.jp}
\begin{abstract}
In this paper, we introduce \textit{splitting numbers} of subvarieties in a smooth complex variety for a Galois cover, 
and prove that the splitting numbers are invariant under certain homeomorphisms. 
In particular cases, we show that splitting numbers enable us to distinguish topologies of complex plane curves even if fundamental groups of complements of plane curves are isomorphic. 
%
Consequently, we prove that there are $\pi_1$-equivalent Zariski $k$-plets for any integer $k\geq2$. 

\end{abstract}

\maketitle

\section*{Introduction}

In this paper, we investigate topologies of plane curves by technique of algebraic geometry, 
where a \textit{topology of a plane curve} means the analytic topology of a pair of the complex projective plane $\bP^2=\bP_{\bC}^2$ and a reduced divisor on $\bP^2$. 
The \textit{combinatorial data} of a plane curve consists of data such as the number of its irreducible components, the degree and the types of singularities of each irreducible component and the intersection data of the irreducible components (see \cite{artal} for details). 
It is known that the combinatorial data of a plane curve determine the topology of tubular neighborhood of the plane curve (cf. \cite{artal}). 
In 1929 \cite{zariski}, O. Zariski proved that the fundamental group of the complement of  a $6$-cuspidal plane sextic is the free product of two cyclic groups of order $2$ and $3$ if the $6$ cusps are on a conic, and the cyclic group of order $6$ otherwise. 
This result \cite{zariski} showed that the combinatorial data of a plane curve does not determine the topology of the plane curve. 
We call a $k$-plet $(C_1,\dots, C_k)$ of plane curves $C_i\subset\bP^2$ a \textit{Zariski $k$-plet} 
if $C_1,\dots,C_k$ have same combinatorial data, and there exist no homeomorphisms $h_{ij}:\bP^2\to\bP^2$ satisfying $h_{ij}(C_j)=C_i$ for any $i$ and $j$ with $i\ne j$. 
In the case of $k=2$, a Zariski $2$-plet is called a \textit{Zariski pair}. 

From 90's, many examples of Zariski pairs were constructed (for example, see \cite{artal1, artal, k-plet, bantoku, oka1, tokunaga}). 
The main tool to distinguish topology of plane curves is fundamental groups as \cite{zariski}. 
Indeed, some topological invariants (for example, Alexander polynomials, characteristic varieties, existence/non-existence of certain Galois covers branched along given curves) are used to distinguish difference of fundamental groups of the complements of plane curves. 
Some authors introduced other methods to distinguish topology of plane curves. 
For example, Artal--Carmona--Cogolludo \cite{braidmonodromy} proved that the braid monodromy of an affine plane curve determine the topology of a related projective plane curve. 
A.~Degtyar\"ev and I.~Shimada constructed Zariski pairs of simple sextic curves by the theory of $K3$-surfaces given by double covers branched at simple sextic curves (for example \cite{degtyarev, degtyarev2, shimada2}). 
Artal--Florens--Guerville \cite{guerville} introduced a new topological invariant, called $\calI$-invariant, for a line arrangement $\calA$ derived from the peripheral structure $\pi_1(B_{\calA})\to \pi_1(E_{\calA})$, 
where $B_{\calA}$ and $E_{\calA}$ are the boundary manifold and the exterior of $\calA$ respectively, 
and constructed Zariski pairs of line arrangements in \cite{guerville, guerville1}. 
Recently, Guerville--Meilhan \cite{guerville2} generalized $\calI$-invariant of line arrangements to an invariant of any algebraic curves, called \textit{linking set}. 
Artal--Tokunaga \cite{k-plet} and Shimada \cite{shimada2} have studied \textit{splitting curves} with respect to double covers to construct Zariski $k$-plet. 
After these works \cite{k-plet, shimada2}, S.~Bannai \cite{bannai} introduced \textit{splitting type} with respect to double covers, 
and he and the author constructed Zariski pairs and $3$-plets by using splitting type in \cite{bannai, banshira}. 
In particular, Degtyar\"ev \cite{degtyarev2} found a Zariski pair $(C_1,\, C_2)$ of simple sextic curves $C_1$ and $C_2$ such that the fundamental groups of the complements of $C_1$ and $C_2$ are isomorphic, such a Zariski pair called a \textit{$\pi_1$-equivalent Zariski pair}. 
Moreover, in \cite{artal}, a Zariski pair $(C_1, C_2)$ such that $\bP^2\setminus C_1$ and $\bP^2\setminus C_2$ are homeomorphic, called a \textit{complement-equivalent Zariski pair}, was constructed by using braid monodromy as \cite{braidmonodromy}. 
However, it seems that there are not many examples of $\pi_1$-equivalent Zariski pairs. 
On the other hand, Shimada \cite{shimada} constructed families of equisingular curves with many connected components, which can not be distinguished by fundamental groups. 
It is not known whether two equisingular curves in distinct connected components of a family of \cite{shimada} provide a Zariski pair, or not. 
In the present paper, we introduce a \textit{splitting number} of an irreducible subvariety in a smooth variety for a Galois cover, which is inspired by the splitting type of Bannai \cite{bannai}, 
and prove that a family of Shimada \cite{shimada} provides a Zariski $k$-plet.

To state the main theorem, we recall the families of Shimada \cite{shimada}. 
In \cite{shimada}, Shimada defined plane curves of type $(b,m)$ as follows. 

\begin{Def}[{\cite[Definition~1.1]{shimada}}]\rm
Let $b$ and $m$ be positive integers such that $b\geq 3$ and $b\equiv 0\pmod{m}$. 
Put $n:=b/m$. 
A projective plane curve $R\subset\bP^2$ is said to be \textit{of type $(b,m)$} if it satisfies the followings; 
\begin{enumerate}
\item 
$R$ consists of two irreducible components $B$ and $E$ of degree $b$ and $3$, respectively, 
\item 
both of $B$ and $E$ are non-singular, 
\item 
the set-theoretical intersection of $B$ and $E$ consists of $3n$ points, and 
\item 
at each intersection point, $B$ and $E$ intersect with multiplicity $m$. 
\end{enumerate} 
\end{Def}

He considered the family $\calF_{b,m}\subset\bP_{\ast}H^0(\bP^2,\calO(b+3))$ of all curves of type $(b,m)$. 
Here $\bP_{\ast}H^0(\bP^2,\calO(d))$ is the projective space of one-dimensional subspaces of the vector space $H^0(\bP^2,\calO(d))$, 
which parameterize all plane curves of degree $d$. 
He proved the following theorem. 

\begin{Th}[{\cite[Theorem~1.2]{shimada}}]\label{th. shimada}
Suppose that $b\geq 4$, and let $m$ be a divisor of $b$. 
\begin{enumerate}
\item 
The number of the connected components of $\calF_{b,m}$ is equal to the number of divisors $m$. 
\item 
Let $R$ be a member of $\calF_{b,m}$. 
Then the fundamental group $\pi_1(\bP^2\setminus R)$ is isomorphic to $\bZ$ if $b$ is not divisible by $3$, while it is isomorphic to $\bZ\oplus\bZ/3\bZ$ if $b$ is divisible by $3$. 
\end{enumerate}
\end{Th}

Hence, if $R_1$ and $R_2$ are curves of type $(b,m)$ in distinct connected components of $\calF_{b,m}$, the embeddings of $R_1$ and $R_2$ in $\bP^2$ can not be distinguished topologically by fundamental groups. 
Note that, if $R_1$ and $R_2$ are in a same connected component of $\calF_{b,m}$, then $(R_1, R_2)$ is not a Zariski pair. 
The following problem was raised. 
\begin{Problem}\rm
For two plane curves $R_1$ and $R_2$ of type $(b,m)$ in distinct connected components of $\calF_{b,m}$, is the pair $(R_1,R_2)$ a Zariski pair? 
\end{Problem}

In the present paper, we introduce a \textit{splitting number} of a subvariety in a smooth variety for a Galois cover (Definition~\ref{def. splitting number}), 
and prove that it is invariant under certain homeomorphisms (Proposition~\ref{lem. splitting number}). 
The main theorem is the following theorem. 

\begin{Th}\label{th. main} 
Let $b\geq 4$ be an integer, and let $m$ be a divisor of $b$, 
and let $R_1$ and $R_2$ be plane curves of type $(b,m)$. 
Then $(R_1,R_2)$ is a Zariski pair if and only if 
$R_1$ and $R_2$ are in distinct connected components of $\calF_{b,m}$. 
\end{Th}

By Theorem~\ref{th. shimada} and \ref{th. main}, we obtain the following corollary. 

\begin{Cor}\label{cor. k-plet}
For any integer $k\geq 2$, there exists a $\pi_1$-equivalent Zariski $k$-plet. 
\end{Cor}

The present paper is organized as follows. 
In the first section, we define splitting numbers of subvarieties in a smooth variety for a Galois cover, and 
prove that splitting numbers are invariant under certain homeomorphisms. 
In the second section, we investigate splitting numbers for simple cyclic covers, and give a method of splitting numbers of smooth plane curves for simple cyclic covers. 
In the third section, we recall the connected components of $\calF_{b,m}$ given in \cite{shimada}. 
In the final section, we prove Theorem~\ref{th. main} and Corollary~\ref{cor. k-plet} by using splitting numbers.

\section{Splitting numbers of subvarieties for Galois covers}

Let $Y$ be a smooth variety. 
Let $B\subset Y$ be a reduced divisor, and let $G$ be a finite group. 
A surjective homomorphism $\theta: \pi_1(Y\setminus B)\twoheadrightarrow G$ induces an \'etale $G$-cover $\phi':X'\to Y\setminus B$, 
where a \textit{$G$-cover} is a Galois cover $\phi:X\to Y$ with $\Gal(\bC(X)/\bC(Y))\cong G$. 
Hence we obtain an extension of rational function fields $\bC(X')/\bC(Y)$. 
By $\bC(X')$-normalization of $Y$, $\pi'$ is extended to a branched $G$-cover $\pi:X\to Y$ uniquely. 
We call $\phi:X\to Y$ the induced $G$-cover by the surjection $\theta:\pi_1(Y\setminus B)\twoheadrightarrow G$. 

\begin{Def}\label{def. splitting number}\rm
Let $Y$ be a smooth variety, and let $\phi:X\to Y$ be an induced $G$-cover branched at a reduced divisor $B\subset Y$ for a finite group $G$. 
For an irreducible subvariety $C\subset Y$ with $C\not\subset B$, we call the number of irreducible components of $\phi^{\ast}C$ the \textit{splitting number of $C$ for $\phi$}, and denote it by $s_{\phi}(C)$. 
If $s_{\phi}(C)\geq 2$, we call $C$ a \textit{splitting subvariety} (a \textit{splitting curve} if $\dim_{\bC} C=1$) for $\phi$. 
\end{Def}

\begin{Rem}\rm
Let $C_1, C_2, B\subset\bP^2$ be three plane curves of degree $d_1,d_2, 2n$, respectively, 
and let $\phi:S'\to\bP^2$ be the double cover branched along $B$. 
Assume that $C_i\not\subset B$ and $s_{\phi}(C_i)=2$ for $i=1,2$, say $\phi^{\ast}C_i=C_i^{+}+C_i^-$. 
In \cite{bannai}, Bannai defined \textit{splitting type} for the triple $(C_1, C_2; B)$ as the pair $(m_1,m_2)$ of the intersection numbers $C_1^+.C_2^+=m_1$ and $C_1^+.C_2^-=m_2$
for suitable choice of labels. 
Moreover, he proved that splitting types are invariant under certain homeomorphisms \cite[Proposition~2.5]{bannai}. 
\end{Rem}

We prove that splitting numbers are invariant under certain homeomorphisms by the same idea of the proof of \cite[Proposition~2.5]{bannai}. 

\begin{Prop}\label{lem. splitting number}
Let $Y$ be a smooth variety. 
Let $B_1, B_2\subset Y$ be two reduced divisors, and let $G$ be a finite group. 
For $i=1,2$, let $\phi_i:X_i\to Y$ be the induced $G$-cover by a surjection $\theta_i:\pi_1(Y\setminus B)\twoheadrightarrow G$. 
Assume that there exists a homeomorphism $h:Y\to Y$ with $h(B_1)=B_2$ and an automorphism $\sigma:G\to G$ and $\sigma\circ\theta_2\circ h_{\ast}=\theta_1$. 
Then the followings hold; 
\begin{enumerate}
\item 
there exists a unique $G$-equivariant homeomorphism $\tilde{h}:X_1'\to X_2'$ with $\phi_2\circ\tilde{h}=h\circ\phi_1$ up to $G$-action, 
where $X_i'=X_i\setminus\phi_i^{-1}(B_i)$ for $i=1,2$, and 
\item 
for two irreducible subvarieties $C_1, C_2\subset Y$ with $C_2=h(C_1)$ and $C_1\not\subset B_1$, 
$\tilde{h}$ induces a one to one correspondence between the set of irreducible components of $\phi_1^{\ast}C_1$ to the set of those of $\phi_2^{\ast}C_2$. 
Moreover, this correspondence is given by 
\[ \widetilde{C}_1\mapsto \overline{\tilde{h}\left(\widetilde{C}_1\setminus \phi_1^{-1}(B_1)\right)} \]
 for an irreducible component $\widetilde{C}_1$ of $\phi_1^{\ast}C_1$, 
 where $\overline{S}$ is the closure of $S$ for a subset $S\subset X_2$. 
 In particular, $s_{\phi_1}(C_1)=s_{\phi_2}(C_2)$. 
\end{enumerate}
\end{Prop}

\begin{proof}
The assertion (1) is clear by the uniqueness of the induced $G$-covers by $\theta_i:\pi_1(Y\setminus B_i)\twoheadrightarrow G$. 
We prove the assertion (2). 
Put $C_i':=C_i\setminus(B_i\cup\Sing(C_i))$ and 
let $\phi_i^{\ast}C_i=\sum_{j=1}^{n_i}\widetilde{C}_{i j}$ be the irreducible decomposition for each $i=1,2$, where $n_i=s_{\phi_i}(C_i)$. 
Since $C_i\setminus (B_i\cup\Sing(C_i))$ is connected 
and $\widetilde{C}'_{i j}:=\widetilde{C}_{i j}\setminus\phi_i^{-1}(B_i\cup\Sing(C_i))$ is smooth, 
$\widetilde{C}'_{i j}$ is a connected component of $\phi_i^{-1}(C_i')$. 
Since $\tilde{h}:X_1'\to X_2'$ is a homeomorphism with $\phi_2\circ\tilde{h}=h\circ\phi_1$, 
$\tilde{h}|_{\phi^{-1}(C_1')}:\phi^{-1}(C_1')\to \phi_2^{-1}(C_2')$ is a homeomorphism. 
Hence the numbers of irreducible components of $\phi_i^{\ast}C_i$ ($i=1,2$) coincide, $n_1=n_2$. 
Moreover, since the closure of $\widetilde{C}'_{i j}$ is $\widetilde{C}_{i j}$, 
the assertion (2) holds. 
\end{proof}

For a smooth curve $B\subset\bP^2$ of degree $b$ and a divisor $m$ of $b$, 
a surjection $\pi_1(\bP^2\setminus B)\twoheadrightarrow \bZ/m\bZ$ is uniquely determined since $\pi_1(\bP^2\setminus B)\cong\bZ/b\bZ$. 
Hence, by Proposition~\ref{lem. splitting number}, we obtain the following corollary. 

\begin{Cor}\label{cor. number}
Let $B_1, B_2, C_1, C_2\subset\bP^2$ be smooth curves with $b=\deg B_1=\deg B_2$, 
and let $m$ be a divisor of $b$. 
Put $\phi_i:X_i\to\bP^2$ the induced  $\bZ/m\bZ$-cover by $\pi_1(\bP^2\setminus B_i)\twoheadrightarrow\bZ/m\bZ$ for each $i=1,2$. 
If there exists a homeomorphism $h: \bP^2\to\bP^2$ with $h(B_1)=B_2$ and $h(C_1)=C_2$, 
then $s_{\phi_1}(C_1)=s_{\phi_2}(C_2)$. 
\end{Cor}

\section{Splitting numbers for simple cyclic covers}\label{sec. simple cyclic cover}

In this section, we investigate a method of computing the splitting number of a smooth plane curve for a simple cyclic cover, where a simple cyclic cover is defined in Definition~\ref{def. simple cyclic cover} below. 
In this section, let $Y$ be a smooth surface. 
For a line bundle $\calL$ on $Y$, 
let $p_{\calL}:T_{\calL}\to Y$ denote the total space associated to $\calL$. 

\begin{Def}\label{def. simple cyclic cover}\rm
Let $B$ be either a reduced curve or the zero divisor on $Y$. 
Assume that there exists a line bundle $\calL$ on $Y$ such that $\calO_Y(B)\cong\calL^{\otimes n}$. 
Let $s\in H^0(Y,\calO_Y(B))$ be a section vanishing exactly along $B$. 
A cyclic cover $\phi: X\to Y$ of degree $n$ is a \textit{simple cyclic cover branched along $B$} if 
$X$ is isomorphic to the subvariety of $T_{\calL}$ defined by $p_{\calL}^{\ast}s-t^n=0$ and 
$\phi$ coincides with the restriction of $p_{\calL}$ to the subviriety, where $t\in H^0(T_{\calL},p_{\calL}^{\ast}\calL)$ is the tautological section. 
\end{Def}

\begin{Def}\rm
Let $\phi:X\to Y$ be a Galois cover branched along $B\subset Y$. 
Let $C\subset Y$ be an irreducible curve of $Y$. 
Let $\widetilde{C}_0$ denote an irreducible component of $\phi^{\ast}C$, and  let $\bar{\eta}_0:\overline{C}_0\to \widetilde{C}_0$ and $\eta:\overline{C}\to C$  be the normalizations. 
We say that $\phi$ is \textit{essentially unramified over $C$} if the induced cover ${\phi}_C:\overline{C}_0\to \overline{C}$ by $\phi\circ\bar{\eta}_0$ is unramified, and \textit{essentially ramified over $C$} otherwise.  
\end{Def}

\begin{Rem}\label{rem. essentially unramified}\rm
Let $G$ be a finite abelian group. 
Let $\phi:X\to Y$ be a $G$-cover over $Y$, and 
let $C$ be an irreducible curve on $Y$. 
The induced cover $\phi_C:\overline{C}_0\to\overline{C}$ is an $G_0$-cover since $G$ is abelian, where $G_0\subset G$ is the stabilizer of $\widetilde{C}_0$. 
If $\phi$ is essentially ramified over $C$, then the quotient $X_1$ of $X$ by the subgroup of $G$, which is generated by all stabilizers of ramification points of $\phi_C:\overline{C}_0\to\overline{C}$, provides an abelian cover $\phi_1:X_1\to Y$ which is essentially unramified over $C$ such that $s_{\phi}(C)=s_{\phi_1}(C)$. 
Hence, to compute the splitting number $s_{\phi}(C)$, we may assume that $\phi$ is essentially unramified over $C$ if $\phi:X\to Y$ is an abelian cover. 
\end{Rem}

Let $\phi:X\to Y$ be a simple cyclic cover of degree $m$ branched along $B\subset Y$, and 
let $C\subset Y$ be an irreducible curve with $C\ne B$. 
We consider the splitting number $s_{\phi}(C)$. 
For an intersection $P\in B\cap C$ and a local branch $\ell$ of $C$ at $P$, 
let $\I_{P,\ell}$ denote the local intersection number of $B$ and $\ell$ at $P$. 
Let $\sigma:\widehat{Y}\to Y$ be a succession of blowing -ups such that $\sigma^{-1}(B+C)$ is simple normal crossing. 
Let $\widehat{C}\subset\widehat{Y}$ denote the strict transformation of $C$ by $\sigma$.  

\begin{Lem}\label{lem. multiplicity}
With the above assumption, let $E_{P,\ell}$ be the irreducible component of $\sigma^{\ast}B$ which intersects with the local branch $\hat{\ell}$ of $\widehat{C}$ corresponding to $\ell$. 
Then the multiplicity of $E_{P,\ell}$ in $\sigma^{\ast}B$ is equal to $I_{P,\ell}$. 
\end{Lem}
\begin{proof}
Since the multiplicity of $E_{P,\ell}$ in $\sigma^{\ast}B$ is depend only on the singularity of $\ell$, 
we may assume that $C$ is locally irreducible at $P$. 
Let $\sigma_P:\widehat{Y}_P\to Y$ be the succession of blowing-ups over $P$ in $\sigma$, 
and let $\widehat{C}_P$ be the strict transform of $C$ by $\sigma_P$.  
Let $E'_{P,\ell}$ be the irreducible component of $\sigma_P^{\ast}B$ which intersects with $\widehat{C}_P$, and 
let $m_{P,\ell}$ be the multiplicity of $E_{P,\ell}$ in $\sigma^{\ast}B$. 
It is sufficient to prove that $m_{P,\ell}=\I_{P,\ell}$. 
Note that the exceptional set $\sigma_P^{-1}(P)$ intersects with $\widehat{C}_P$ at one point. 
By the projection formula, we obtain 
\begin{align*}
m_{P,\ell} 
&= \sigma_P^{\ast}B.\widehat{C}_P-\sum_{(P',\ell')\ne (P,\ell)}\I_{P',\ell'} \\
&= B.C-\sum_{(P',\ell')\ne (P,\ell)}\I_{P',\ell'} \\
&= \I_{P,\ell}. 
\qedhere
\end{align*}
\end{proof}

We may assume that $E.\widehat{C}\leq1$ for each irreducible component $E$ of $\sigma^{\ast}B$ after more blowing-ups if necessary. 
Let $\hat{\phi}:\widehat{X}\to\widehat{Y}$ be the $\bC(X)$-normalization of $\widehat{Y}$. 
In general, $\hat{\phi}$ is not a simple cyclic cover. 
By Lemma~\ref{lem. multiplicity}, if there is a local branch $\ell$ of $C$ at $P\in B\cap C$ such that $m$ is not a divisor of $\I_{P,\ell}$, 
then $\hat{\phi}$ is branched along $E_{P,\ell}$, 
hence $\phi$ is essentially ramified over $C$. 
By Remark~\ref{rem. essentially unramified}, we may assume 
\[ \I_{P,\ell}\equiv 0\pmod{m} \]
for any local branch $\ell$ of $C$ at $P\in B\cap C$. 
Let $L$ be a divisor on $Y$ whose associated line bundle $\calL$ defines $\phi:X\to Y$ as in Definition~\ref{def. simple cyclic cover}, 
and let $D_{B,C}$ denote the following divisor on $\widehat{C}$; 
\[ D_{B,C}:=\frac{1}{m}(\sigma^{\ast} B){|_{\widehat{C}}}=\left.\left(\frac{1}{m}\sum_{(P,\ell)}\I_{P,\ell}E_{P,\ell}\right)\right|_{\widehat{C}}, \]
where the summand runs over all intersections $P\in B\cap C$ and all local branches $\ell$ of $C$ at $P$. 
Put $D_{B,C}':=(\sigma^{\ast}L){|_{\widehat{C}}}-D_{B,C}$. 
Note that, by regarding $\widehat{C}$ as the smooth model of $C$, $D_{B,C}$ does not depend on choice of $\sigma$ by Lemma~\ref{lem. multiplicity}. 
Moreover, $mD'_{B,C}$ is linearly equivalent to $0$ on $\widehat{C}$.

\begin{Prop}\label{prop. splitting number for simple}
Under the above circumstance, $s_{\phi}(C)=\nu$ if and only if the order of $[\calO_{\widehat{C}}(D'_{B,C})]\in\Pic^0(\widehat{C})$ is equal to $m/\nu$. 
\end{Prop}
\begin{proof}
We put 
\[ \widehat{\calL}:=\calO_{\widehat{Y}}(\sigma^{\ast}L-\frac{1}{m}\sum\I_{P,\ell}E_{P,\ell}). \]
Let $\hat{\phi}:\widehat{X}\to\widehat{Y}$ be the $\bC(X)$-normalization of $\widehat{Y}$, and put 
\[ \widehat{Y}':=\widehat{Y}\setminus\Supp(\sigma^{\ast}B-\sum\I_{P,\ell}E_{P,\ell}) \ \mbox{ and } \ \widehat{X}':=\hat{\phi}^{-1}(\widehat{Y}'). \]
Note that $\widehat{C}\subset\widehat{Y}'$. 
The restriction of $\hat{\phi}$ to $\widehat{X}'$, $\hat{\phi}':\widehat{X}'\to \widehat{Y}'$, is an \'etale simple cyclic cover of degree $m$ defined in $T_{\widehat{\calL}}$ over $\widehat{Y}'$. 
By Stein factorization of $\sigma\circ\hat{\phi}:\widehat{X}\to Y$, we obtain a birational morphism $\tilde{\sigma}:\widehat{X}\to X$ with $\phi\circ\tilde{\sigma}=\sigma\circ\hat{\phi}$. 
The birational morphism $\tilde{\sigma}$ provides a one to one correspondence between irreducible components of $\hat{\phi}^{\ast}\widehat{C}$ and those of $\phi^{\ast}C$. 
Thus we have $s_{\phi}(C)=s_{\hat{\phi}}(\widehat{C})=s_{\hat{\phi}'}(\widehat{C})$. 
Since the restriction of $T_{\widehat{\calL}}$ over $\widehat{C}$ is $T_{\calO(D'_{B,C})}$, the assertion follows from the next lemma. 
\end{proof}

\begin{Lem}
Let $C$ be a smooth variety, and let $\calL$ be a line bundle on $C$ with $\calL^{\otimes \mu}\cong\calO_C$ and $\calL^{\otimes i}\not\cong\calO_C$ for $1\leq i<\mu$. 
Put $m:=\mu\nu$ for some $\nu\in\bZ_{>0}$. 
Let $\widetilde{C}$ be the closed subset of $T_{\calL}$ defined by $t^m-1=0$, 
where $t\in H^0(T_{\calL},p_{\calL}^{\ast}\calL)$ is the tautological section. 
Then the number of connected components of $\widetilde{C}$ is equal to $\nu$. 
\end{Lem}
\begin{proof}
Since $\calL^{\otimes \mu}\cong\calO_C$, $t^\mu-\zeta_{\nu}^j=0$ defines closed subset of $T_{\calL}$ for $1\leq j< \nu$, where $\zeta_\nu$ is a primitive $\nu$-th root of unity. 
Since $\calL^{\otimes i}\not\cong\calO_C$ for $1\leq i<\mu$, $t^i-a=0$ does not define globally a closed subset of $T_\calL$. 
Thus the number of connected components of $\widetilde{C}$ is equal to $\nu$. 
\end{proof}

In general, it seems difficult to compute the order of $[\calO_{\widehat{C}}(D'_{B,C})]\in\Pic^0(\widehat{C})$. 
However, the following theorem provides a method of computing splitting numbers of smooth plane curves for simple cyclic covers. 

\begin{Th}
Let $\phi:X\to\bP^2$ be a simple cyclic cover of degree $m$ branched along a plane curve $B$ of degree $b=mn$, and 
let $C\subset\bP^2$ be a smooth curve of degree $d$. 
Assume that $\I_P\equiv 0\pmod{m}$ for each $P\in B\cap C$, 
where $\I_P$ is the local intersection multiplicity of $B$ and $C$ at $P$. 
Let $\nu$ be a divisor of $m$, say $m=\mu\nu$. 
Then 
$s_{\phi}(C)=\nu$ if and only if the following conditions hold; 
\begin{enumerate}
\item 
for $1\leq k<\mu$, there are no curves $D_{kn}\subset\bP^2$ of degree $kn$ such that $D_{kn}|_C= k D_{B,C}$, 
where $D_{B,C}$ is regarded as a divisor on $C$; 
\item
there exists a curve $D_{\mu n}\subset\bP^2$ of degree $\mu n$ such that $D_{\mu n}|_C=\mu D_{B,C}$. 
\end{enumerate}
\end{Th}

\begin{proof}
Let $f=0$ be a defining equation of $C\subset\bP^2$. 
We have the following exact sequence; 
\[ 0\to H^0(\bP^2,\calO_{\bP^2}(kn-d))\overset{\times f}{\to} H^0(\bP^2,\calO_{\bP^2}(kn))\overset{\alpha}{\to} H^0(C,\calO_C(kn))\to 0, \]
where $\alpha$ is the restriction to $C$. 
By Proposition~\ref{prop. splitting number for simple}, $s_\phi(C)=\nu$ if and only if the order of $[\calO_{C}(D'_{B,C})]\in\Pic^0(C)$ is $\mu$. 
The surjection $\alpha$ implies that the order of $[\calO_{C}(D'_{B,C})]$ is $\mu$ if and only if the conditions (1) and (2) hold. 
\end{proof}

\section{Plane curves of type $(b,m)$}

In this section, we recall the equisingular families of plane curves given in \cite{shimada}. 
Let $b$ be a positive integer with $b\geq3$, and let $m$ be a divisor of $b$. 
We put $n:=b/m$. 
Let $\calF_{b,m}\subset \bP_{\ast} H^0(\bP^2,\calO(b+3))$ be the family of all curves of type $(b,m)$. 
Note that any two curves $R$ and $R'$ of type $(b,m)$ have same combinatorics. 
Let $R=B+E$ be a curve of type $(b,m)$. 
Let $D_R$ denote the reduced divisor $(B|_E)_{\mathrm{red}}$ of degree $3n$ on $E$, 
and let $H$ denote a divisor of degree $3$ on $E$ that is obtained as the intersection of $E$ and a line on $\bP^2$. 
Then $\calO_E(D_R-n H)$ is an invertible sheaf of degree $0$ on $E$; 
\[ [\calO_E(D_R-nH)]\in\Pic^0(E). \]
Note that $D_R$ and $D_R-nH$ correspond to $D_{B,E}$ and $D_{B,E}'$ in Section~\ref{sec. simple cyclic cover}, respectively. 
Let $\lambda(R)$ be the order of the isomorphism class $[\calO_E(D_R-nH)]$ in $\Pic^0(E)$, 
which is a divisor of $m$. 
For a divisor $\mu$ of $m$, we write by $\calF_{b,m}(\mu)$ the union of all connected components of $\calF_{b,m}$ on which the function $\lambda$ is constantly equal to $\mu$. 
\[ \calF_{b,m}=\coprod_{\mu | m}\calF_{b,m}(\mu) \]
Shimada \cite{shimada} proved the following proposition. 

\begin{Prop}[{\cite[Proposition~2.1]{shimada}}]\label{prop. shimada}
Suppose that $b\geq 3$, and let $m$ be a divisor of $b$. 
For any divisor $\mu$ of $m$, the variety $\calF_{b,m}(\mu)$ is irreducible and of dimension $(b-1)(b-2)/2+3n+8$. 
\end{Prop}

\begin{Rem}\rm
Proposition~\ref{prop. shimada} implies that the number of connected components of $\calF_{b,m}$ is equal to the number of divisors of $m$. 
\end{Rem}

\section{Proofs}

In this section, we prove Theorem~\ref{th. main} and Corollary~\ref{cor. k-plet}. 

\begin{proof}[Proof of Theorem~\ref{th. main}]
Let $\mu_1$ and $\mu_2$ be distinct divisors of $m$, 
and let $R_i=B_i+E_i$ be a member of $\calF_{b,m}(\mu_i)$ for each $i=1,2$. 
Let $\phi_i:X_i\to\bP^2$ be the simple cyclic cover of degree $m$ branched along $B_i$ for each $i=1,2$. 
If there exists a homeomorphism $h:\bP^2\to\bP^2$ such that $h(R_1)=R_2$, 
then $h(B_1)=B_2$ and $h(E_1)=E_2$ since $\deg B_i=b>3=\deg E_i$. 
By Corollary~\ref{cor. number}, we obtain $s_{\phi_1}(E_1)=s_{\phi_2}(E_2)$. 
On the other hand, by Proposition~\ref{prop. splitting number for simple} and the definition of $\calF_{b,m}(\mu_i)$, we have $s_{\phi_i}(E_i)=m/\mu_i$ for each $i=1,2$. 
Hence $(R_1,R_2)$ is a Zariski pair. 

Conversely, if $\mu_1=\mu_2$, then it is clear that there is a homeomorphism $h:\bP^2\to\bP^2$ such that $h(R_1)=R_2$ since $R_1$ and $R_2$ are members of $\calF_{b,m}(\mu_1)=\calF_{b,m}(\mu_2)$. 
\end{proof}

\begin{proof}[Proof of Corollary~\ref{cor. k-plet}]
By Theorem~\ref{th. shimada}, the number of connected components of $\calF_{5^{k-1},5^{k-1}}$ is equal to $k$. 
For each integer $0\leq i\leq k-1$, let $R_{i}$ be a member of $\calF_{5^{k-1},5^{k-1}}(5^i)$.  
By Theorem~\ref{th. shimada} and \ref{th. main}, $(R_0,\dots,R_{k-1})$ is a $\pi_1$-equivalent Zariski $k$-plet. 
\end{proof}

\end{document}